\documentclass{amsart}

\usepackage[initials]{amsrefs}
\usepackage{amssymb,amsmath,amsfonts,latexsym,amsthm,geometry,graphicx}

\newtheorem{theorem}{Theorem}[section]
\newtheorem{corollary}[theorem]{Corollary}
\newtheorem{proposition}[theorem]{Proposition}
\newtheorem{lemma}[theorem]{Lemma}

\theoremstyle{definition}

\newtheorem{remark}[theorem]{Remark}

\newcommand{\bpolkm}{\mathrm B^{\rm pol}_{\mathbb K,m}}
\newcommand{\bmultkm}{\mathrm B^{\rm mult}_{\mathbb K,m}}
\newcommand{\bpolcm}{\mathrm B^{\rm pol}_{\mathbb C,m}}
\newcommand{\bmultcm}{\mathrm B^{\rm mult}_{\mathbb C,m}}

\newcommand{\bmultrm}{\mathrm B^{\rm mult}_{\mathbb R,m}}
\newcommand{\bmultc}[1]{\mathrm B^{\rm mult}_{\mathbb C,#1}}
\newcommand{\bmultk}[1]{\mathrm B^{\rm mult}_{\mathbb K,#1}}

\newcommand{\Mmn}{\mathcal M(m,n)}
\newcommand{\Jmn}{\mathcal J(m,n)}
\newcommand{\bi}{\mathbf i}
\newcommand{\Pkm}[1]{\mathcal P_{#1}(m)}
\newcommand{\veps}{\varepsilon}
\newcommand{\Arp}[1]{\mathrm A_{\mathbb R,#1}}
\newcommand{\Acp}[1]{\mathrm A_{\mathbb C,#1}}
\newcommand{\Akp}[1]{\mathrm A_{\mathbb K,#1}}

\newcommand{\TT}{\mathbb T}

\begin{document}

\title[The Bohr radius of $\mathbb D^n$]{The Bohr radius of the $n$-dimensional polydisk\\ is equivalent to $\sqrt{(\log n)/n}$}

\date{}

\author[Bayart]{Fr\'{e}d\'{e}ric Bayart}
\address{Laboratoire de Math\'{e}matique, \newline \indent
Universit\'{e} Blaise Pascal Campus des C\'{e}zeaux, \newline \indent
F-63177 Aubiere Cedex, France.}
\email{Frederic.Bayart@math.univ-bpclermont.fr}

\author[Pellegrino]{Daniel Pellegrino}
\address{Departamento de Matem\'{a}tica, \newline \indent
Universidade Federal da Para\'{i}ba, \newline \indent
58.051-900 - Jo\~{a}o Pessoa, Brazil.}
\email{dmpellegrino@gmail.com and pellegrino@pq.cnpq.br}

\author[Seoane]{Juan B. Seoane-Sep\'{u}lveda}
\address{Departamento de An\'{a}lisis Matem\'{a}tico,\newline\indent Facultad de Ciencias Matem\'{a}ticas, \newline\indent Plaza de Ciencias 3, \newline\indent Universidad Complutense de Madrid,\newline\indent Madrid, 28040, Spain.}
\email{jseoane@mat.ucm.es}

\keywords{Bohr radius; Interpolation; Bohnenblust--Hille inequality}
%\subjclass[2010]{46G25, 30B50}

\begin{abstract}
We show that the Bohr radius of the polydisk $\mathbb D^n$ behaves asymptotically as $\sqrt{(\log n)/n}$. Our argument is based on a new interpolative approach to the Bohnenblust--Hille inequalities which allows us to prove, among other results, that the polynomial Bohnenblust--Hille inequality is subexponential.
\end{abstract}

\maketitle

%\thanks{D. Pellegrino was supported by CNPq Grant 477124/2012-7, INCT-Matem\'{a}tica and CAPES-NF.}

\section{Introduction}

Following Boas and Khavinson \cite{BK97}, the Bohr radius $K_n$ of the $n$-dimensional polydisk
is the largest positive number $r$ such that all polynomials
$\sum_{\alpha}a_\alpha z^\alpha$ on $\mathbb C^n$ satisfy
$$\sup_{z\in r\mathbb D^n}\sum_{\alpha}|a_\alpha z^{\alpha}|\leq \sup_{z\in\mathbb D^n}\left|\sum_{\alpha}a_\alpha z^\alpha\right|.$$
The Bohr radius $K_1$ was studied and estimated by H. Bohr himself, and it was shown independently by M. Riesz, I. Schur and F. Wiener that $K_1=1/3$.
For $n\geq 2$, exact values or $K_n$ are unknown. However, in \cite{BK97}, the two inequalities
\begin{equation}\label{EQBOASKHA}
 \frac 13\sqrt{\frac 1n}\leq K_n\leq 2\sqrt{\frac{\log n}n}
\end{equation}
were established.

The paper of Boas and Khavinson was a source of inspiration for many subsequent papers, linking the asymptotic behaviour of $K_n$ to various problems in
functional analysis (geometry of Banach spaces, unconditional basis constant of spaces of polynomials, etc.), we refer to \cite{DP06} for a survey of some of them.
Hence there was a big interest in recent years in determining the behaviour of $K_n$ for large values of $n$.

In \cite{DeFr06}, the left inequality of (\ref{EQBOASKHA}) was improved to $K_n\geq c\sqrt{\log n/(n\log\log n)}$.
In \cite{DFOOS10}, using the hypercontractivity of the polynomial Bohnenblust--Hille inequality, the authors showed that
$$K_n=b_n\sqrt{\frac{\log n}n}\textrm{ with }\frac1{\sqrt 2}+o(1)\leq b_n\leq 2.$$
Our first main result is the exact asymptotic behaviour of $K_n$. More precisely, we prove that

$$K_n\sim_{+\infty}\sqrt{\frac{\log n}n}.$$

\smallskip

The main tool used to prove this, the second main result of this paper, is a substantial improvement of the polynomial Bohnenblust--Hille inequality. Recall that the multilinear Bohnenblust--Hille inequality (see \cite{bh}) asserts that, for any $m\geq 1$, there exists a constant $C_m\geq 1$ such that,
for all $m$-linear forms $L:c_0\times\dots\times c_0\to\mathbb K$,
$$
\left(  \sum\limits_{i_{1},\ldots,i_{m}=1}^{\infty}\left\vert L(e_{i_{^{1}}%
},\ldots,e_{i_{m}})\right\vert ^{\frac{2m}{m+1}}\right)  ^{\frac{m+1}{2m}}\leq
C_{m}\left\Vert L\right\Vert.
$$
The optimal constant $C_m$ in this inequality will be denoted by $\bmultkm$. The case $m=2$ is nothing else but the famous Littlewood's 4/3 inequality.

Using polarization, Bohnenblust and Hille also obtained a polynomial version of this inequality:
for any $m\geq 1$, there exists a constant $D_m\geq 1$ such that, for any $n\geq 1$, for any $m$-homogeneous polynomial $P(z)=\sum_{|\alpha|=m}a_\alpha z^{\alpha}$ on $\mathbb C^n$,
$$\left(\sum_{|\alpha|=m}|a_{\alpha}|^{\frac{2m}{m+1}}\right)^{\frac{m+1}{2m}}\leq D_m \|P\|_\infty,$$
where $\|P\|_\infty=\sup_{z\in\mathbb D^n}|P(z)|$. In turn, the best constant $D_m$ in this inequality will be denoted by $\bpolkm$.

\smallskip

These inequalities have been proven to be very useful and powerful in analysis: for instance, to estimate the abscissae of convergence of Dirichlet series (this was the initial goal of Bohnenblust and Hille), to study the behavior of power series in several complex variables (this is the so-called Bohr radius problem mentioned above) or in quantum physics (see \cite{monta}). In these applications, it turns out that having good estimates of the constants $\bpolkm$ and $\bmultkm$ is crucial.

There are several proofs of the Bohnenblust-Hille inequalities, some of which are presented in \cite{surv}. Very recently, the authors gave in \cite{ba} yet another one, based on interpolation. It leads to the following enhancement:  if $m\geq1$ and $q_{1},\ldots,q_{m}\in\lbrack1,2],$ then the
following assertions are equivalent:

\begin{itemize}
\item[(A)] There is a constant $\mathrm B_{\mathbb K,{q_{1},\ldots, q_{m}}}\geq1$ such that{\small {
\begin{equation}\label{EQBHGENERALIZED}
\left(  {\textstyle\sum\limits_{i_{1}=1}^{\infty}}\left(  {\textstyle\sum
\limits_{i_{2}=1}^{\infty}}\left(  \ldots \left(  {\textstyle\sum\limits_{i_{m-1}%
=1}^{\infty}}\left(  {\textstyle\sum\limits_{i_{m}=1}^{\infty}}\left\vert
L\left(  e_{i_{1}},\ldots,e_{i_{m}}\right)  \right\vert ^{q_{m}}\right)
^{\frac{q_{m-1}}{q_{m}}}\right)  ^{\frac{q_{m-2}}{q_{m-1}}}\cdots\right)
^{\frac{q_{2}}{q_{3}}}\right)  ^{\frac{q_{1}}{q_{2}}}\right)  ^{\frac{1}%
{q_{1}}}\leq \mathrm B_{\mathbb K,{q_{1},\ldots ,q_{m}}}\left\Vert L\right\Vert%
\end{equation}
}}for all continuous $m$-linear forms $L:c_{0}\times\cdots\times
c_{0}\rightarrow\mathbb{K}$.

\item[(B)] $\frac{1}{q_{1}}+\cdots+\frac{1}{q_{m}}\leq\frac{m+1}{2}.$
\end{itemize}
Compared to all previous known proofs of the Bohnenblust--Hille inequality, the proof in  \cite{ba} is probably the simplest. However, it gives bad constants: indeed, they have an
exponential growth in the extremal case $q_{1}=\cdots=q_{m}=\frac{2m}{m+1}$, and this is a little bit {\em disappointing} at a first glance, having in mind that the optimal constants of the multilinear Bohnenblust--Hille inequality
have a subpolynomial growth (see \cite{jfalimite}).

In this paper we improve the best known estimates on $\bpolkm$ and $\bmultkm$. To explain our results, let us give a short historical account of the improvements on the control of the growth of $\bpolkm$ and $\bmultkm$ for the case of complex scalars:
\begin{itemize}
\item for the multilinear case, $\bmultcm\leq m^{\frac{m+1}{2m}}(\sqrt 2)^{m-1}$ (Bohnenblust and Hille, 1931),
$\bmultcm\leq (\sqrt 2)^{m-1}$ (Davie \cite{Dav73} in 1973, Kaijser \cite{Ka} in 1978),  $\bmultcm\leq \left(\frac 2{\sqrt \pi}\right)^{m-1}$ (Queff\'{e}lec \cite{Qu95} in 1995).
A big step was done after the publication of \cite{def} on separately multiple summing mappings: $\bmultcm$ is subpolynomial (\cite{jfalimite} in 2013). The smallest known upper bound is
$$\bmultcm\leq \frac 2{\sqrt \pi} m^{\log_2{e^{\frac 12-\frac\gamma 2}}},$$
where $\gamma$ is the Euler-Mascheroni constant (this can be proved by refining the argument used in \cite{jfalimite}).
\item for the polynomial case, polarization shows that
$$\bpolcm\leq \bmultcm\frac{m^{\frac m2}(m+1)^{\frac{m+1}2}}{2^m (m!)^{\frac{m+1}{2m}}}.$$
Avoiding a direct use of polarization, a much better estimate was obtained in \cite{DFOOS10} (2011):
\begin{equation}\label{EQBHDEFANT}
\bpolcm\leq \left(1+\frac 1{m-1}\right)^{m-1}\sqrt m(\sqrt 2)^{m-1}.
\end{equation}
We show that we can go further: $\bpolcm$ is, actually, subexponential!
\end{itemize}

\begin{theorem}\label{THMMAINPOL}
For every $\kappa_2>1$, there exists $\kappa_1>0$ such that, for any $m\geq 1$,
$$\bpolcm\leq\kappa_1\exp(\kappa_2\sqrt{m\log m}).$$
\end{theorem}

Regarding the multilinear Bohnenblust--Hille inequality, we are able to improve the best upper bound by showing that
$$\bmultcm\leq \kappa m^{\frac{1-\gamma}2}.$$
Moreover, our proof is rather more elementary than the proof given in \cite{jfalimite}.

\medskip

Many modern proofs of the Bohnenblust-Hille inequalities depend on variants of an inequality due to Blei (see \cite{Bl79}). The proofs of these
inequalities are rather technical. Using interpolation, as in \cite{ba}, we shall give (in Section \ref{SECBLEI}) new proofs of these inequalities, providing also improvements of them.
 These improvements are the basis to our new upper bounds for the  Bohnenblust-Hille constants, which can be found in Sections \ref{SECMULTI} and \ref{SECPOL}.
 In Section \ref{SECOTHER}, we show that if $(q_1,\dots,q_m)$ is close to the extremal exponent $(\frac{2m}{m+1},\dots,\frac{2m}{m+1})$, then the best $\mathrm B_{\mathbb K,{q_{1},\ldots, q_{m}}}$ in (\ref{EQBHGENERALIZED}) has also a subpolynomial growth. Finally, in Section \ref{SECBOHR}, we apply Theorem \ref{THMMAINPOL} to obtain the precise asymptotic behaviour of the $n$-dimensional Bohr radius.

\section{A new interpolative insight of Blei's inequalities}\label{SECBLEI}
We need to introduce some notations. For two positive integers $n,m$, we set
 \begin{eqnarray*}
\Mmn&=&\big\{\bi=(i_1,\dots,i_m); \ i_1,\dots,i_m\in\{1,\dots,n\}\big\}\\
\Jmn&=&\big\{\bi\in\Mmn;\ i_1\leq i_2\leq\dots\leq i_m\big\}.
\end{eqnarray*}
For $1\leq k\leq m$, let $\Pkm{k}$ denote the set of subsets of $\{1,\dots,m\}$ with cardinal $k$. For $S=\{s_1,\dots,s_k\}$ in $\Pkm{k}$,
$\hat S$ will denote its complement in $\{1,\dots,m\}$ and $\bi_S$ shall mean $(i_{s_1},\dots,i_{s_k})\in\mathcal M(k,n)$.

For $\mathbf q=(q_1,\dots,q_m)\in [1,+\infty)^m$, we shall consider the Lorentz space $\ell_{\mathbf q}=\ell_{q_1}\big(\ell_{q_2}(\dots(\ell_{q_m}(\mathbb N)))\big)$, namely $(a_{\bi})\in\ell_{\mathbf q}$ if and only if
$$\left(  {\textstyle\sum\limits_{i_{1}=1}^{\infty}}\left(  {\textstyle\sum
\limits_{i_{2}=1}^{\infty}}\left(  \ldots \left(  {\textstyle\sum\limits_{i_{m-1}%
=1}^{\infty}}\left(  {\textstyle\sum\limits_{i_{m}=1}^{\infty}}\left\vert
a_{\bf i}\right\vert ^{q_{m}}\right)
^{\frac{q_{m-1}}{q_{m}}}\right)  ^{\frac{q_{m-2}}{q_{m-1}}}\cdots\right)
^{\frac{q_{2}}{q_{3}}}\right)  ^{\frac{q_{1}}{q_{2}}}\right)  ^{\frac{1}%
{q_{1}}}<+\infty.$$
We will interpolate between Lorentz spaces. It is well-known (see \cite{ba,berg}) that if $\mathbf p,\mathbf q\in [1,+\infty)^m$ and
$\theta\in (0,1)$, then
$$[\ell_{\mathbf p},\ell_{\mathbf q}]_{\theta}=\ell_{\mathbf r}$$
with $\frac 1{r_i}=\frac{\theta}{p_i}+\frac{1-\theta}{q_i}$ for $1\leq i\leq m$.

\smallskip

Our second main tool is a consequence of Minkowski's inequality which can be found in, e.g., \cite[Corollary 5.4.2]{garling}: for any $0<p\leq q<+\infty$ and for any sequence of complex  numbers $(c_{i,j})$,
$$\left(\sum_i\left(\sum_j |c_{i,j}|^p\right)^{q/p}\right)^{1/q}\leq \left(\sum_j\left(\sum_i |c_{i,j}|^q\right)^{p/q}\right)^{1/p}.$$
In particular, let $S\in\Pkm{k}$, let $\lambda\in[1,2]$ and let $\mathbf q=(q_1,\dots,q_m)$ with $q_i=\lambda$ if $i\in S$, $q_i=2$ otherwise. Then an easy induction shows that, for any family of complex numbers $(a_{\bf i})_{\bi\in\Mmn}$,
$$\|a\|_{\ell_\mathbf q}\leq \left(\sum_{\bi_S}\left(\sum_{\bi_{\hat S}} |a_{\bi}|^2\right)^{\frac\lambda 2}\right)^{\frac 1\lambda}$$
(see also \cite[Proposition 3.1]{ba}), the symbol $\sum_{\bi_S}$ meaning $\sum_{i_{s_1}=1}^n \dots\sum_{i_{s_k}=1}^n$.

\smallskip

Blei's inequality (see \cite{DFOOS10}) states that, for all families of complex numbers $(a_{\bf i})_{\bi\in\Mmn}$,  we have
$$\left(\sum_{\bi \in\Mmn}|a_{\bi}|^{\frac{2m}{m+1}}\right)^{\frac{m+1}{2m}}\leq \prod_{1\leq j\leq m}\left(\sum_{i_j=1}^n\left(
\sum_{\stackrel{i_1,\dots,i_{j-1},}{i_{j+1},\dots,i_m=1}}^n |a_{\bi}|^2\right)^{1/2}\right)^{1/m}.$$
With our notations,
$$\left(\sum_{\bi \in\Mmn}|a_{\bi}|^{\frac{2m}{m+1}}\right)^{\frac{m+1}{2m}}\leq \prod_{S\in\Pkm{1}}\left(\sum_{\mathbf i_S}\left(
\sum_{\mathbf i_{\hat S}} |a_{\bi}|^2\right)^{1/2}\right)^{1/m}.$$
As an application of our interpolative approach, we generalize this inequality by replacing $\Pkm{1}$ by any $\Pkm{k}$, $1\leq k\leq m$. This result will be crucial later.
\begin{theorem}\label{THMBLEIGENERALIZED}
Let $m,n\geq 1$ and $1\leq k\leq m$. Then for all families $(a_{\bi})_{\bi\in\Mmn}$ of complex numbers,
$$\left(\sum_{\bi \in\Mmn}|a_{\bi}|^{\frac{2m}{m+1}}\right)^{\frac{m+1}{2m}}\leq \prod_{S\in\Pkm{k}}\left(\sum_{\mathbf i_S}\left(
\sum_{\mathbf i_{\hat S}} |a_{\bi}|^2\right)^{\frac12\times\frac{2k}{k+1}}\right)^{\frac{k+1}{2k}\times\frac 1{\binom mk}}.$$
\end{theorem}
\begin{proof}
For $S\in\Pkm{k}$, let $\mathbf q^S$ be defined by $q_i^S=\frac{2k}{k+1}$ provided $i\in S$ and $q_i^S=2$ otherwise. Let also
$\mathbf q=(\frac{2m}{m+1},\dots,\frac{2m}{m+1})$ and $\theta=\frac 1{\binom mk}$. Then for any $i\in\{1,\dots,m\}$,
$$\frac{1}{q_i}=\sum_{S\in\Pkm{k}}\frac{\theta}{q_i^S}.$$
Indeed, by symmetry, $q_i=q_j$ for any $i\neq j$ and for any $S\in\Pkm{k}$,
$$\sum_{i=1}^m \frac{1}{q_i^S}=k\times\frac{k+1}{2k}+\frac{m-k}2=\frac{m+1}2.$$
Hence,
$$\frac{m}{q_i}=\sum_{j=1}^m\frac1{q_j}=\sum_{j=1}^m\sum_{S\in\Pkm{k}}\frac{\theta}{q_j^S}=\frac{m+1}2.$$
Thus, by interpolation,
$$\|a\|_{\ell_\mathbf q}\leq \prod_{S\in\Pkm{k}}\|a\|_{\ell_{\mathbf q^S}}^{\frac 1{\binom mk}}.$$
The left-hand side of this inequality is exactly $\left(\sum_{\bi} |a_{\bi}|^{\frac{2m}{m+1}}\right)^{\frac{m+1}{2m}}$ whereas, for any $S\in\Pkm{k}$,
$$\|a\|_{\ell_{\mathbf q}^S}\leq \left(\sum_{\bi_S}\left(\sum_{\bi_{\hat S}} |a_\bi|^2\right)^{\frac{1}2\times\frac{2k}{k+1}}\right)^{\frac{k+1}{2k}}.$$
\end{proof}
\begin{remark}
In Theorem \ref{THMBLEIGENERALIZED}, we can consider other exponents. The same proof shows that, if $p,q,s$ are bigger than $1$ with $q\geq s$ and
$$\frac ks+\frac{m-k}q=\frac mp,$$
then for all families of complex numbers $(a_{\bf i})_{\bi\in\Mmn}$,  we have
$$\left(\sum_{\bi \in\Mmn}|a_{\bi}|^{p}\right)^{\frac{1}{p}}\leq \prod_{S\in\Pkm{k}}\left(\sum_{\mathbf i_S}\left(
\sum_{\mathbf i_{\hat S}} |a_{\bi}|^q\right)^{\frac sq}\right)^{\frac{1}{s}\times\frac 1{\binom mk}}.$$
\end{remark}
Our interpolation arguments are also useful to prove the following variant of Blei's inequality, which was the starting point of \cite{def}.

\begin{theorem}
[Defant, Popa, Schwarting]\label{b} Let $A$ and $B$ be two finite non-void index
sets. Let $(a_{ij})_{(i,j)\in A\times B}$ be a scalar matrix with positive
entries, and denote its columns by $\alpha_{j}=(a_{ij})_{i\in A}$ and its rows
by $\beta_{i}=(a_{ij})_{j\in B}.$ Then, for $q,s_{1},s_{2}\geq1$ with
$q>\max(s_{1},s_{2})$ we have
\[
\left(  \sum_{(i,j)\in A\times B}a_{ij}^{w(s_{1},s_{2})}\right)  ^{\frac
{1}{w(s_{1},s_{2})}}\leq\left(  \sum_{i\in A}\left\Vert \beta_{i}\right\Vert
_{q}^{s_{1}}\right)  ^{\frac{f(s_{1},s_{2})}{s_{1}}}\left(  \sum_{j\in
B}\left\Vert \alpha_{j}\right\Vert _{q}^{s_{2}}\right)  ^{\frac{f(s_{2}%
,s_{1})}{s_{2}}},
\]
with
\begin{align*}
w  &  :[1,q)^{2}\rightarrow\lbrack0,\infty),\text{ }w(x,y):=\frac
{q^{2}(x+y)-2qxy}{q^{2}-xy},\\
f  &  :[1,q)^{2}\rightarrow\lbrack0,\infty),\text{ }f(x,y):=\frac{q^{2}%
x-qxy}{q^{2}(x+y)-2qxy}.
\end{align*}
\end{theorem}

\begin{proof}
Let us consider the exponents $\left(  q,\ldots,q,s_{2},\ldots,s_{2}\right)
,\left(  s_{1},\ldots,s_{1},q,\ldots,q\right)  $ and $\left(  \theta
_{1},\theta_{2}\right)  =\left(  f(s_{2},s_{1}),f(s_{1},s_{2})\right)  $. Note
that $w\left(  s_{1},s_{2}\right)  $ is obtained by interpolating $\left(
s_{2},q\right)  $ and $\left(  q,s_{1}\right)  $ with $\theta_{1},\theta_{2}$,
respectively. Then%
\[
\left(  \sum_{(i,j)\in A\times B}a_{ij}^{w(s_{1},s_{2})}\right)  ^{\frac
{1}{w(s_{1},s_{2})}}\leq\left(  \sum_{i\in A}\left\Vert \beta_{i}\right\Vert
_{q}^{s_{1}}\right)  ^{\frac{f(s_{1},s_{2})}{s_{1}}}\left(  \sum_{i\in
A}\left\Vert \beta_{i}\right\Vert _{s_{2}}^{q}\right)  ^{\frac{f(s_{2},s_{1}%
)}{q}}.
\]
All that is left to prove is that the order of the last sum can be changed, but this is true because $q\geq s_{2}$.
\end{proof}

The above approach also stresses that these inequalities are just  particular
cases of a huge family of similar inequalities that can be proved by an
analogous interpolative procedure.

\section{The multilinear Bohnenblust-Hille inequalities}\label{SECMULTI}
We now investigate the multilinear Bohnenblust-Hille inequalities. An important tool to obtain them is Khintchine inequality.
Let $(\veps_i)$ be a sequence of independent Rademacher variables. Then, for any $p\in [1,2]$, there exists a constant
$\Arp{p}$ such that, for any $n\geq 1$ and any $a_1,\dots,a_n\in\mathbb R$,
$$\left(\sum_{i=1}^n |a_i|^2\right)^{1/2}\leq \mathrm A_{\mathbb R,p}\left(\int_{\Omega}\left|\sum_{i=1}^n a_i\veps_i(\omega)\right|^p d\omega\right)^{1/p}.$$
It has a complex counterpart: for any $p\in [1,2]$, there exists a constant
$\Acp{p}$ such that, for any $n\geq 1$ and any $a_1,\dots,a_n\in\mathbb C$,
$$\left(\sum_{i=1}^n |a_i|^2\right)^{1/2}\leq \mathrm A_{\mathbb C,p}\left(\int_{\TT^n}\left|\sum_{i=1}^n a_i z_i\right|^p dz\right)^{1/p}.$$
The best constant $\Arp{p}$ and $\Acp{p}$ are known (see \cite{Haa} and \cite{KoKw}). Indeed,
\begin{itemize}
\item $\Arp{p}=\frac1{\sqrt 2}\left(\frac{\Gamma\left(\frac{1+p}2\right)}{\sqrt \pi}\right)^{-1/p}$ if $p>p_0\approx 1.847$;
\item $\Acp{p}=\Gamma\left(\frac{p+2}2\right)^{-1/p}$ if $p\in (1,2]$.
\end{itemize}
Using Fubini's theorem and Minkowski's inequality (see, for instance, \cite[Lemma 2.2]{def} for the real case and \cite[Theorem 2.2]{jfa22} for the complex case), these inequalities have a multilinear version:
for any $n,m\geq 1$, for any family $(a_{\bi})_{\bi\in\Mmn}$ of real (resp. complex) numbers,
$$\left(\sum_{\bi\in\Mmn}^n |a_\bi|^2\right)^{1/2}\leq \mathrm A_{\mathbb R,p}^m \left(\int_{\Omega}\left|\sum_{\bi\in\Mmn} a_\bi\veps_{i_1}^{(1)}(\omega)\dots\veps_{i_m}^{(m)}(\omega)\right|^p d\omega\right)^{1/p}$$
where $(\veps_i^{(1)}),\dots,(\veps_i^{(m)})$ are independent sequences of independent Rademacher variables (resp.
$$\left(\sum_{\bi\in\Mmn}^n |a_\bi|^2\right)^{1/2}\leq \mathrm A_{\mathbb C,p}^m \left(\int_{\TT^{nm}}\left|\sum_{\bi\in\Mmn} a_\bi z_{i_1}^{(1)}\dots z_{i_m}^{(m)} \right|^p dz^{(1)}\dots dz^{(m)}\right)^{1/p},$$
in the complex case).

\medskip

We are ready to give an inductive formula for $\bmultkm$.
\begin{proposition}\label{PROPINDUCMULT}
For any $m\geq 2$, for any $1\leq k\leq m-1$,
$$\bmultkm\leq \Akp{\frac{2k}{k+1}}^{m-k}\bmultk{k}.$$
\end{proposition}
\begin{proof}
We just consider the complex case. Let $n\geq 1$ and let $L=\sum_{\bi\in\Mmn}a_{\bi}z_{i_1}^{(1)}\dots z_{i_m}^{(m)}$
be an $m$-linear form on $\mathbb C^n$. By Theorem \ref{THMBLEIGENERALIZED}, it suffices to prove that, for any $S\in\Pkm{k}$,
$$\left(\sum_{S}\left(\sum_{\hat S}|a_{\bi}|^2\right)^{\frac 12\times \frac{2k}{k+1}}\right)^{\frac{k+1}{2k}}\leq \Acp{\frac{2k}{k+1}}^{m-k}\bmultc{k}\|L\|.$$
For the sake of clarity, we shall assume that $S=\{1,\dots,k\}$. For any $i_1,\dots,i_k\in\{1,\dots,n\}$, by the multilinear Khintchine inequality,
$$\left(\sum_{\hat S}|a_{\bi}|^2\right)^{\frac 12}\leq \Acp{\frac{2k}{k+1}}^{m-k}\left(\int_{\TT^{n(m-k)}}\left|\sum_{i_{k+1},\dots,i_m=1}^n a_{\bi} z_{i_{k+1}}^{(k+1)}\dots z_{i_m}^{(m)}\right|^{\frac{2k}{k+1}}dz^{(k+1)}\dots dz^{(m)}\right)^{\frac{k+1}{2k}}.$$
But for a fixed choice of $z^{(k+1)},\dots,z^{(m)}$, we know that
\begin{eqnarray*}
\sum_{i_1,\dots,i_k}\left|\sum_{i_{k+1},\dots,i_m}a_\bi z_{i_{k+1}}^{(k+1)}\dots z_{i_m}^{(m)}\right|^{\frac{2k}{k+1}}&\leq&\left(\bmultc{k}\sup_{z^{(1)},\dots,z^{(k)}\in\mathbb T^n}\left|\sum_{\bi}a_\bi z_{i_1}^{(1)}\dots z_{i_m}^{(m)}\right|\right)^{\frac{2k}{k+1}}\\
&\leq&\left(\bmultc{k}\|L\|\right)^{\frac{2k}{k+1}}.
\end{eqnarray*}
Hence,
\begin{eqnarray*}
\sum_{S}\left(\sum_{\hat S}|a_\bi|^2\right)^{\frac12\times\frac{2k}{k+1}}&\leq&\Acp{\frac{2k}{k+1}}^{(m-k)\frac{2k}{k+1}}\int_{\TT^{n(m-k)}}\sum_{i_1,\dots,i_k}\left|\sum_{i_{k+1},\dots,i_m}a_\bi z_{i_{k+1}}^{(k+1)}\dots z_{i_m}^{(m)}\right|^{\frac{2k}{k+1}}dz^{(k+1)}\dots dz^{(m)}\\
&\leq&\Acp{\frac{2k}{k+1}}^{(m-k)\frac{2k}{k+1}}\int_{\TT^{n(m-k)}}\left(\bmultc{k}\|L\|\right)^{\frac{2k}{k+1}}dz^{(k+1)}\dots dz^{(m)}\\
&\leq&\left(\Acp{\frac{2k}{k+1}}^{(m-k)}\bmultc{k}\|L\|\right)^{\frac{2k}{k+1}}.
\end{eqnarray*}
\end{proof}
When $m$ is even and $k=m/2$, we obtain
$$\bmultkm\leq \Akp{\frac{2m}{m+2}}^{m/2}\bmultk{m/2}.$$
This formula was used in \cite{jfalimite} to obtain the subpolynomial growth of $\bmultkm$. However, it seems unknown
if the sequence $(\bmultkm)$ is nondecreasing, and the method of \cite{jfalimite} was rather involved.
Using a better choice of $k$, we will get a much simpler proof with better estimates.
\begin{corollary}\label{CORMULTIC}
There exists $\kappa>0$ such that, for any $m\geq 1$,
$$\bmultcm\leq \kappa m^{\frac{1-\gamma}2}.$$
\end{corollary}
Numerically, $\frac{1-\gamma}2\simeq 0.21392$
\begin{proof}
We apply Proposition \ref{PROPINDUCMULT} with $k=m-1$. Then
$$\bmultcm\leq \Gamma\left(2-\frac1 m\right)^{-\frac{m}{2m-2}}\bmultc{m-1}.$$
Now, $\Gamma(2)=1$ and $\Gamma'(2)=1-\gamma$, thus
\begin{eqnarray*}
\Gamma\left(2-\frac1 m\right)^{-\frac{m}{2m-2}}&=&\exp\left(\left(-\frac12+O\left(\frac 1m\right)\right)\log\left(1-\frac{1-\gamma}m+O\left(\frac 1{m^2}\right)\right)\right)\\
&=&\exp\left(\frac{1-\gamma}{2m}+O\left(\frac 1{m^2}\right)\right).
\end{eqnarray*}
This easily yields
$$\bmultcm\leq \kappa\exp\left(\frac{1-\gamma}2\log m\right)=\kappa m^{\frac{1-\gamma}2}.$$
\end{proof}
\begin{corollary}
There exists $\kappa>0$ such that, for any $m\geq 1$,
$$\bmultrm\leq \kappa m^{\frac{2-\ln 2-\gamma}2}.$$
\end{corollary}
Numerically, $\frac{2-\gamma-\ln 2}2\simeq 0.36481$.
\begin{proof}
The proof is completely similar, using that, for $m$ sufficiently large,
$$\Arp{\frac{2m-2}m}=\frac 1{\sqrt 2}\left(\frac{\Gamma\left(\frac 32-\frac 1m\right)}{\sqrt \pi}\right)^{-\frac m{2m-2}}$$
and that $\Gamma(3/2)=\frac{\sqrt \pi}2$, $\Gamma'(3/2)=\frac{\sqrt \pi}2(2-2\ln 2-\gamma)$.
\end{proof}

\begin{corollary}
$\limsup_{m\to+\infty}(\bmultkm-\bmultk{m-1})=0.$
\end{corollary}
\begin{proof}
By Proposition \ref{PROPINDUCMULT},
$$\bmultkm-\bmultk{m-1}\leq \bmultk{m-1}\left(\Akp{\frac{2m-2}m}-1\right).$$
Now, $\bmultk{m-1}=o(m)$ whereas it is easy to check that $\Akp{\frac{2m-2}m}-1=O\left(\frac1m\right)$.
\end{proof}

The recursive formulas obtained above can be easily written as closed formulas. For instance,

\[
\bmultcm\leq%
%TCIMACRO{\dprod \limits_{j=2}^{m}}%
%BeginExpansion
{\displaystyle\prod\limits_{j=2}^{m}}
%EndExpansion
\Gamma\left(  2-\frac{1}{j}\right)  ^{\frac{j}{2-2j}}.%
\]

For the real case, \textit{mutatis mutandis}, a similar formula can be obtained.

\section{Other exponents}\label{SECOTHER}

In this section, we study the value of the best constant $\mathrm B_{\mathbb K,q_1,\dots,q_m}$ in (\ref{EQBHGENERALIZED}).
A multi-index $(q_1,\dots,q_m)$ with $\frac{1}{q_1}+\dots+\frac{1}{q_m}=\frac{m+1}2$ will be called a Bohnenblust-Hille exponent.

\medskip

The essence of Corollary \ref{CORMULTIC} is just the following: from the multilinear Khintchine inequality, we know that
the constant associated to the Bohnenblust-Hille exponent $\left(\frac{2m-2}m,\dots,\frac{2m-2}m,2\right)$ is dominated by
$\Acp{\frac{2m-2}m}\bmultc{m-1}$. Varying the position of the power 2 in $\left(\frac{2m-2}m,\dots,\frac{2m-2}m,2\right)$
we still have the upper bound $\Acp{\frac{2m-2}m}\bmultc{m-1}$. Interpolating the $n$ Bohnenblust-Hille exponents
 $\left(\frac{2m-2}m,\dots,\frac{2m-2}m,2,\frac{2m-2}m,\dots,\frac{2m-2}m\right)$ with $\theta_1=\dots=\theta_n$,
we obtain the Bohnenblust-Hille exponent $\left(\frac{2m}{m+1},\dots,\frac{2m}{m+1}\right)$ with the same constant $\Acp{\frac{2m-2}m}\bmultc{m-1}$.
Now, we can change the way we interpolate (by changing the values of $\theta_1,\dots,\theta_n$) and we can also put
several 2 instead of just one. This leads to the following.

\begin{proposition}
 For any $K>0$, there exists $\kappa>0$ such that, for any $m\geq 1$, for any Bohnenblust-Hille exponent $(q_1,\dots,q_m)$
with $\left|q_i-\frac{2m}{m+1}\right|\leq \frac Km$,
then
$$\mathrm B_{\mathbb C,q_1,\dots,q_m}\leq\kappa m^{\frac{1-\gamma}2}.$$
\end{proposition}
\begin{proof} There is no loss of generality in working with sufficiently large values of $m$.
Let $\lambda \in(0,1)$ be such that, for any $m\geq 1$ large enough, one can find $k$ in $[\lambda m,m]$ such that
$$\frac{2m}{m+1}-\frac{2k}{k+1}>\frac Km.$$
By the multilinear Khintchine inequality, we know that
$$\mathrm B_{\mathbb C,\frac{2k}{k+1},\dots,\frac{2k}{k+1},2,\dots,2}\leq \Acp{\frac{2k}{k+1}}^{m-k}\bmultc{k},$$
where, in $\mathrm B_{\mathbb C,\frac{2k}{k+1},\dots,\frac{2k}{k+1},2,\dots,2}$, the exponent $\frac{2k}{k+1}$ appears $k$ times.
Since $k\geq\lambda m$, a rapid look at the value of $\Acp{\frac{2k}{k+1}}$ shows that $\Acp{\frac{2k}{k+1}}^{m-k}$ is bounded
by some constant which does not depend on $m$. Hence,
$$\mathrm B_{\mathbb C,\frac{2k}{k+1},\dots,\frac{2k}{k+1},2,\dots,2}\leq \kappa m^{\frac{1-\gamma}2}.$$
We still keep the same upper bound for every Bohnenblust-Hille exponent $(q_1,\dots,q_m)$ with
$q_i\in\frac{2k}{k+1}$ for $k$ values of $i$ and $q_i=2$ for the $(m-k)$ other values of $i$. Now, if we interpolate
between these exponents, we get
$$\mathrm B_{\mathbb C,q_1,\dots,q_m}\leq\kappa m^{\frac{1-\gamma}2}$$
for every Bohnenblust-Hille coefficient $(q_1,\dots,q_m)$ with $q_i\in\left[\frac{2k}{k+1},2\right]$.
The proposition then follows straightforwardly.
\end{proof}

\section{The polynomial Bohnenblust-Hille inequality is subexponential}\label{SECPOL}
Let us turn to the polynomial Bohnenblust-Hille inequality. We need a polynomial version of the Khintchine inequality.
It can be found in \cite[Theorem 9]{BAYMONAT}.
\begin{lemma}\label{LEMMONAT}
 Let $p\in[1,2]$. For every $m$-homogeneous polynomial $\sum_{|\alpha|=m}a_{\alpha}z^{\alpha}$ on $\mathbb C^n$,
we have
$$\left(\sum_{|\alpha|=m}|a_\alpha|^2\right)^{1/2}\leq\left(\frac 2p\right)^{m/2}\left\|\sum_{|\alpha|=m}a_\alpha z^{\alpha}\right\|_{L^p(\TT^n)}.$$
\end{lemma}
We start from $P=\sum_{|\alpha|=m}a_\alpha z^\alpha$ an $m$-homogeneous polynomial on $\mathbb C^n$. We shall also write it
$$P(z)=\sum_{\bi\in\Jmn}c_\bi z_{i_1}\dots z_{i_m}.$$ There exists a symmetric $m$-multilinear form $L:\mathbb C^n\times\dots\times\mathbb C^n\to\mathbb C$
so that $L(z,\dots,z)=P(z)$. To define $L$, we need to introduce some standard notations. For indices $\bi,\mathbf j\in\Mmn$, the notation
$\bi\sim\mathbf j$ means that there is a permutation $\sigma$ of the set $\{1,\dots,m\}$ such that $i_{\sigma(k)}=j_k$ for every $k=1,\dots,m$.
For a given index $\bi$, we denote by $[\bi]$ the equivalence class of all indices $\mathbf j$ such that $\mathbf j\sim\bi$.
Moreover, let $|\bi|$ denote the cardinality of $[\bi]$.  Note that for each $\bi\in\Mmn$, there is a unique $\mathbf j\in\Jmn$ with $[\bi]=[\mathbf j]$.

The symmetric $m$-multilinear form $L$ is then defined by
$$L(z^{(1)},\dots,z^{(m)})=\sum_{\bi\in\Mmn}\frac{c_{[\bi]}}{|\bi|}z_{i_1}^{(1)}\cdots z_{i_m}^{(m)}.$$
The norm of $L$ is controlled by $\frac{m^m}{m!}\|P\|_\infty$. We will need a more precise control, when we evaluate $L$ on products of diagonals. By a result of Harris (\cite{Ha72}),
for any nonnegative integers $m_1,\dots,m_k$ with $m_1+\dots+m_k=m$, for any $z^{(1)},\dots,z^{(k)}\in\mathbb T^n$,
$$|L(\underbrace{z^{(1)},\dots,z^{(1)}}_{m_1},\dots,\underbrace{z^{(k)},\dots,z^{(k)}}_{m_k})|\leq\frac{m_1!\cdots m_k!}{m_1^{m_1}\cdots m_k^{m_k}}\times\frac{m^m}{m!}\|P\|_\infty.$$

Our main step is the following inductive inequality linking
$\bpolcm$ and $\bmultcm$.
\begin{theorem}\label{THMPOLMULT}
 For any $m\geq 2$, for any $1\leq k\leq m-1$,
$$\bpolcm\leq \left(1+\frac 1k\right)^{\frac{m-k}2}\times\frac{m^m}{(m-k)^{m-k}}\times\left(\frac{(m-k)!}{m!}\right)^{1/2}\bmultc{k}.$$
\end{theorem}
\begin{proof}
 Keeping the same notations,
\begin{eqnarray*}
\sum_{|\alpha|=m}|a_\alpha|^{\frac{2m}{m+1}}&=&\sum_{\bi\in\Jmn}|c_{\bi}|^{\frac{2m}{m+1}}\\
&=&\sum_{\bi\in\Mmn}|\bi|^{-\frac1{m+1}}\left(\frac{|c_{[\bi]}|}{|\bi|^{1/2}}\right)^{\frac{2m}{m+1}}\\
&\leq&\sum_{\bi\in\Mmn} \left(\frac{|c_{[\bi]}|}{|\bi|^{1/2}}\right)^{\frac{2m}{m+1}}.
\end{eqnarray*}
We then apply Theorem \ref{THMBLEIGENERALIZED} to obtain
$$\left(\sum_{|\alpha|=m}|a_\alpha|^{\frac{2m}{m+1}}\right)^{\frac{m+1}{2m}}\leq\prod_{S\in\Pkm{k}}\left(\sum_{\bi_S}\left(\sum_{\bi_{\hat S}}\frac{|c_{[\bi]}|^2}{|\bi|}\right)^{\frac12\times\frac{2k}{k+1}}\right)^{\frac{k+1}{2k}\times\frac1{\binom mk}}.$$
It is easy to check that for any $S\in\Pkm{k}$, $\frac{|\bi|}{|\bi_{\hat S}|}\leq m(m-1)\cdots (m-k+1)$. Thus,
$$\left(\sum_{|\alpha|=m}|a_\alpha|^{\frac{2m}{m+1}}\right)^{\frac{m+1}{2m}}\leq\big(m(m-1)\cdots(m-k+1)\big)^{1/2}\prod_{S\in\Pkm{k}}\left(\sum_{\bi_S}\left(\sum_{\bi_{\hat S}}\frac{|c_{[\bi]}|^2}{|\bi|^2}|\bi_{\hat S}|\right)^{\frac12\times\frac{2k}{k+1}}\right)^{\frac{k+1}{2k}\times\frac1{\binom mk}}.$$
We fix some $S\in\Pkm{k}$. As before, we may and shall assume for the sake of clarity that $S=\{1,\dots,k\}$. We then fix some $\bi_S\in\mathcal M(k,n)$ and we introduce
the following $(m-k)$-homogeneous polynomial on $\mathbb C^n$:
$$P_{\bi_S}(z)=L(e_{i_1},\dots,e_{i_k},z,\dots,z).$$
Observe that
\begin{eqnarray*}
 P_{\bi_S}(z)&=&\sum_{\bi_{\hat S}\in\mathcal M(m-k,n)} \frac{c_{[\bi]}}{|\bi|}z_{i_{k+1}}\cdots z_{i_m}\\
&=&\sum_{\bi_{\hat S}\in\mathcal J(m-k,n)}\frac{c_{[\bi]}}{|\bi|} |\bi_{\hat S}| z_{i_{k+1}}\cdots z_{i_m}
\end{eqnarray*}
so that
\begin{eqnarray*}
 \|P_{\bi_S}\|_2&=&\left(\sum_{\bi_{\hat S}\in\mathcal J(m-k,n)} \frac{|c_{[\bi]}|^2}{|\bi|^2}|\bi_{\hat S}|^2\right)^{1/2}\\
&=&\left(\sum_{\bi_{\hat S}\in\mathcal M(m-k,n)} \frac{|c_{[\bi]}|^2}{|\bi|^2}|\bi_{\hat S}|\right)^{1/2}.
\end{eqnarray*}
By Lemma \ref{LEMMONAT},
$$\|P_{\bi_S}\|_2^{\frac{2k}{k+1}}\leq \left(1+\frac 1k\right)^{\frac{(m-k)k}{k+1}}\int_{\mathbb T^{n}}\big|L(e_{i_1},\dots,e_{i_k},z,\dots,z)\big|^{\frac{2k}{k+1}}dz.$$
This leads to
$$\sum_{\bi_S}\left(\sum_{\bi_{\hat S}}\frac{|c_{[\bi]}|^2}{|\bi|^2}|\bi_{\hat S}|\right)^{\frac12\times\frac{2k}{k+1}}\leq \left(1+\frac 1k\right)^{\frac{(m-k)k}{k+1}}\int_{\mathbb T^{n}}
\sum_{\bi_S}\big|L(e_{i_1},\dots,e_{i_k},z,\dots,z)\big|^{\frac{2k}{k+1}}dz.$$
Now, for a fixed $z\in\mathbb T^n$, we may apply the multilinear Bohnenblust-Hille inequality to the $k$-multilinear form
$$(z^{(1)},\dots,z^{(k)})\mapsto L(z^{(1)},\dots,z^{(k)},z,\dots,z).$$
Then
\begin{eqnarray*}
 \sum_{\bi_S}\big|L(e_{i_1},\dots,e_{i_k},z,\dots,z)\big|^{\frac{2k}{k+1}} &\leq&\left(\bmultc{k} \sup_{w^{(1)},\dots,w^{(k)}\in\mathbb T^n} \big|L(w^{(1)},\dots,w^{(k)},z,\dots,z)\big|\right)^{\frac{2k}{k+1}}\\
&\leq&\left(\bmultc{k}\times\frac{(m-k)!}{(m-k)^{m-k}}\times\frac{m^m}{m!}\|P\|_\infty\right)^{\frac{2k}{k+1}}.
\end{eqnarray*}
Summarizing all the previous estimates, we finally get
$$\left(\sum_{|\alpha|=m}|a_\alpha|^{\frac{2m}{m+1}}\right)^{\frac{m+1}{2m}}\leq \big(m(m-1)\cdots(m-k+1)\big)^{1/2} \left(1+\frac 1k\right)^{\frac{m-k}{2}}\frac{(m-k)!m^m}{(m-k)^{m-k}m!}\bmultc{k}
\|P\|_\infty$$
as announced.
\end{proof}
When $k=1$, we recover the inequality of \cite{DFOOS10}. Observe also that the case $k=m$ would correspond to a direct use of the polarization inequality. Our task now is to find the best value of $k$.
\begin{corollary}\label{CORPOL1}
 For any $\veps>0$, there exists $\kappa>0$ such that, for any $m\geq 1$,
$$\bpolcm\leq \kappa(1+\veps)^m.$$
\end{corollary}
\begin{proof}
 Let $k\geq 1$ be such that $\frac1k\leq\veps$. Then the result follows from
$$\left(1+\frac 1k\right)^{\frac{m-k}2}\leq (1+\veps)^{m/2} \quad \text{ and }  \quad \frac{m^m}{(m-k)^{m-k}}\times\left(\frac{(m-k)!}{m!}\right)^{1/2}=O(m^{k/2}).$$
\end{proof}
\begin{corollary}
 For any $\kappa_2>1$, there exists $\kappa_1>0$ such that
$$\bpolcm\leq \kappa_1\exp(\kappa_2 \sqrt{m\log m}).$$
\end{corollary}
\begin{proof}
 We set $k=\sqrt m/{\sqrt{\log m}}$. Then
$$\left(1+\frac 1k\right)^{\frac{m-k}2}\leq \exp\left(\frac{m}{2k}\right)=\exp\left(\frac{\sqrt{m\log m}}{2}\right)$$
whereas, by Stirling's formula,
\begin{eqnarray*}
 \frac{m^m}{(m-k)^{m-k}}\times\left(\frac{(m-k)!}{m!}\right)^{1/2}&\leq&\kappa e^{k/2} \frac{m^{m/2}}{(m-k)^{\frac{m-k}2}}\\
&\leq&\kappa e^{k/2} \left(\frac m{m-k}\right)^{\frac{m-k}2}m^{k/2}\\
&\leq&\kappa e^{k/2} \exp\left(\frac{k-m}2\log\left(1-\frac km\right)\right)m^{k/2}\\
&\leq&\exp\left(\frac{\sqrt{m\log m}}2+o\big(\sqrt{m\log m}\big)\right).
\end{eqnarray*}
\end{proof}

\section{Application: the exact asymptotic behaviour of the Bohr radius}\label{SECBOHR} We now prove that $$K_n\sim_{+\infty}\sqrt{\frac{\log n}n}.$$
As we said in the Introduction, in \cite{DFOOS10}, using (\ref{EQBHDEFANT}), the authors show that
\begin{equation}
K_n=b_n\sqrt{\frac{\log n}n}\textrm{ with }b_n\geq \frac1{\sqrt 2}+o(1).\label{EQBOHRDEFANT}
\end{equation}
Replacing (\ref{EQBHDEFANT}) by Corollary \ref{CORPOL1}, we obtain (with the same proof) the estimate
$$K_n=b_n\sqrt{\frac{\log n}n}\textrm{ with }b_n\geq 1+o(1).$$
Since the proof of (\ref{EQBOHRDEFANT}) is only sketched in \cite{DFOOS10}, we nevertheless give the details
of this estimate. We shall need the following lemma due to F. Wiener (see \cite{DFOOS10}):

\begin{lemma}
\label{LEMWIENER} Let $P$ be a polynomial in $n$ variables and $P=\sum
_{m\geq0}P_{m}$ its expansion in homogeneous polynomials. If $\left\Vert P\right\Vert _{\infty}\leq1$, then $\left\Vert P_{m}\right\Vert
_{\infty}\leq1-\left\vert P_{0}\right\vert ^{2}$ for all $m>0.$
\bigskip
\end{lemma}

We begin with a polynomial $\sum_{\alpha}a_{\alpha}z^{\alpha}$ such that
$\sup_{z\in\mathbb{D}^{n}}\left\vert \sum_{\alpha}a_{\alpha}z^{\alpha
}\right\vert \leq1$. Observe that for all $z\in r\mathbb{D}^{n}$,
\[
\sum_{\alpha}|a_{\alpha}z^{\alpha}|\leq|a_{0}|+\sum_{m\geq1}r^{m}\sum
_{|\alpha|=m}|a_{\alpha}|.
\]
Let $\varepsilon>0$ and let $m_{0}\geq1$ be very large (its value will depend
on $\varepsilon$). We set $\displaystyle r=(1-2\varepsilon)\sqrt{\frac{\log n}{n}}$. Using
Wiener's lemma, Corollary \ref{CORPOL1}, and H\"{o}lder's inequality, we obtain
\[
\sum_{m\geq1}r^{m}\sum_{|\alpha|=m}|a_{\alpha}|\leq\kappa(1-|a_{0}|^{2})\sum_{m\geq1}\big(r(1+\varepsilon)\big)^{m}\binom{n+m-1}{m}^{\frac{m-1}{2m}}.
\]
Now,
\[
\binom{n+m-1}{m}\leq e^{m}\left(  1+\frac{n}{m}\right)  ^{m}
\]
so that
\[
\sum_{m\geq1}r^{m}\sum_{|\alpha|=m}|a_{\alpha}|\leq\kappa(1-|a_{0}|^{2}%
)\sum_{m\geq1}\big(r\sqrt{e}(1+\varepsilon))^{m}\left(  1+\frac{n}{m}\right)
^{\frac{m-1}{2}}.
\]
We now split the sum into several parts. First, when $m>\sqrt{n}$, then
\[
\left(  1+\frac{n}{m}\right)  ^{\frac{m-1}{2}}\leq(2\sqrt{n})^{m/2}
\]
which yields
\[
\sum_{m>\sqrt{n}}\big(r\sqrt{e}(1+\varepsilon)\big)^{m}\left(  1+\frac{n}{m}\right)  ^{\frac{m-1}{2}}\leq\sum_{m>\sqrt{n}}\left(  \sqrt{\frac{\log
n}{n}}\sqrt{2e}(1-2\varepsilon)(1+\varepsilon)n^{1/4}\right)  ^{m}
\]
and this quantity goes to zero as $n$ goes to infinity. When $m\leq\sqrt{n}$,
we just write, provided $n$ is large enough and $m\geq m_{0}$,
\[
\left(  1+\frac{n}{m}\right)  ^{\frac{m-1}{2}}\leq(1+\varepsilon)^{\frac{m}{2}}\frac{n^{m/2}}{n^{1/2}m^{m/2}}.
\]
Thus,
\[
\sum_{m_{0}\leq m\leq\sqrt{n}}r^{m}\sum_{|\alpha|=m}|a_{\alpha}|\leq\kappa(1-|a_{0}|^{2})\sum_{m_{0}\leq m\leq\sqrt{n}}\left(  (1-2\varepsilon
)(1+\varepsilon)^{\frac{3}{2}}\sqrt{\frac{\log n}{n}}\sqrt{e}\sqrt{n}\sqrt{\frac{1}{n^{1/m}m}}\right)  ^{m}.
\]
The function $m\mapsto n^{1/m}m$ is decreasing until $m=\log n$, where its
value is equal to $e\log n$, and increasing after $\log n$. This
implies
\[
\sum_{m_{0}\leq m\leq\sqrt{n}}r^{m}\sum_{|\alpha|=m}|a_{\alpha}|\leq
\kappa(1-|a_{0}|^{2})\sum_{m_{0}\leq m\leq\sqrt{n}}\left(  (1-2\varepsilon
)(1+\varepsilon)^{\frac{3}{2}}\right)  ^{m}=o(1)
\]
provided $m_{0}$ is large enough (with respect to $\varepsilon$). For the
remaining part of the sum, we use that if $\log n\geq m_{0},$ then
$m\mapsto n^{1/m}m$ is decreasing in $\left[  1,m_{0}\right]  $, so,
\[
\sum_{m=1}^{m_{0}}r^{m}\sum_{|\alpha|=m}|a_{\alpha}|\leq\kappa(1-|a_{0}%
|^{2})\sum_{m=1}^{m_{0}}\left(  \frac{(1-2\varepsilon)(1+\varepsilon
)^{\frac{3}{2}}m^{1/2m}\sqrt{e}\sqrt{\log n}}{m_{0}^{1/2}n^{1/{2m_{0}}}%
}\right)  ^{m}=o(1).
\]

\bigskip

Summing these three inequalities, we have obtained that
\[
\sum_{m\geq0}r^{m}\sum_{|\alpha|=m}|a_{\alpha}|\leq|a_{0}|+(1-|a_{0}%
|^{2})o(1)\leq1.
\]
Hence, $K_{n}\geq(1-2\varepsilon)\sqrt{\log n}/\sqrt{n}$ provided $n$ is large enough.

\smallskip

That $\limsup_{n\to+\infty}K_n\sqrt{n/\log n}\leq 1$ has already been observed in \cite{BK97}. For the sake of completeness, we recall the method. By the Kahane-Salem-Zygmund inequality,
there exist coefficients $(c_\alpha)_{|\alpha|=m}$ with $|c_\alpha|=\binom{m}{\alpha}$ such that
$$\left\|\sum_{|\alpha|=m}c_\alpha z^{\alpha}\right \|_\infty\leq \kappa \sqrt{m\log m}(m!)^{1/2}n^{\frac{m+1}2}.$$
Thus, we know that
$$K_n^m n^m=\sum_{|\alpha|=m}|c_\alpha| K_n^m\leq\left\|\sum_{|\alpha|=m}c_\alpha z^{\alpha}\right \|_\infty\leq\kappa\sqrt{m\log m}(m!)^{1/2}n^{\frac{m+1}2}$$
so that
$$K_n\leq \kappa^{1/m}(\sqrt{m\log m})^{1/m}\frac{1}{\sqrt n}n^{1/2m}(m!)^{1/2}.$$
We now choose $m=\log n$ (not surprisingly, we need the same relation between the number of variables and the degree of the polynomial!) and, then, use Stirling's formula in order to obtain the desired estimate.

\vspace{.5cm}

\noindent \textbf{Acknowledgements}. D. Pellegrino thanks CAPES for the financial support. The authors also thank Prof. Daniel Cariello for fruitful conversation and comments on the manuscript.

%%%%%%%%%%
%REFERENCES %
%%%%%%%%%%

\end{document}